\documentclass[11pt]{article}

\usepackage{amsmath}
\usepackage{amssymb}
\usepackage{bbm}
\usepackage{amsfonts}
\usepackage{amsthm}
\usepackage{parskip}
\usepackage{tikz}

\newtheorem{thm}{Theorem}[section]
\newtheorem{prop}[thm]{Proposition}

\newtheorem{lem}[thm]{Lemma}
\newtheorem*{defi}{Definition}

\title{Algebraic Stability of Oscillatory Integral Estimates: A Calculus for Uniform Estimates}
\author{John Green}
\date{}

\begin{document}
\maketitle

\begin{abstract}
Oscillatory integrals arise in many situations where it is important to obtain decay estimates which are stable under certain perturbations of the phase. Examining the structural problems underpinning these estimates leads one to consider sublevel set estimates, which behave nicely under certain algebraic operations such as composition with a polynomial. This motivates us to ask how oscillatory integral estimates behave under such transformations of the phase, and under some natural higher order convexity assumptions we obtain stable estimates under composition with polynomial phases in one dimension, and in higher dimensions in the setting of the higher dimensional van der Corput's lemma of Carbery-Christ-Wright \cite{carbery1999multidimensional}.
\end{abstract}

\begin{footnotesize}
\end{footnotesize}
\section{Introduction}
We will be interested in how oscillatory integrals of the form
$$I(\lambda)=\int e^{i\lambda f(x)}\,dx$$
decay for large values of a real ``frequency parameter" $\lambda$, where $f$ is a real-valued ``phase" function. Typically this decay will be expressed as
$$|I(\lambda)|\leq C\lambda^{-\delta}$$
for some $\delta>0$ (one may also be interested in other expressions for the decay, for instance, including a logarithmic factor, for simplicity we do not pursue this).

In addition to the decay rate, one may also be interested in uniformity of the constant $C$ within some class of phases, so that such estimates are stable under perturbations of the phase within a given set. This situation occurs, for instance, when examining the Fourier transform of some density on a hypersurface in $\mathbb{R}^n$ expressed as a graph $(x,\phi(x))$ in terms of $|\xi|$, where $\xi$ is the Fourier variable. In this case we consider the phases $(x,\phi(x))\cdot\omega$ for $\omega\in S^{n-1}$.

In connection with the related sublevel set estimates (more on this shortly), we will be interested in questions of this sort: Can we obtain oscillatory integral estimates for phases resulting from algebraic operations on phase functions for which estimates are known? And if so, how do the decay rate and constants depend on those of our original estimates?

We will consider some basic examples of this, but our main results concern the following:

\textbf{Question. }Let $f$ be a phase function such that oscillatory integral estimates of the form
$$\left|\int e^{i\lambda f(x)}\,dx\right|\leq C|\lambda|^{-\delta}$$
hold. Is it true that
$$\left|\int e^{i\lambda P(f(x))}\,dx\right|\leq C'|\lambda|^{-\delta/d}$$
for some class of normalised polynomials $P$ of degree $d$, and can $C'$ be taken to be uniform in the relevant parameters ($C$, $d$, etc), independently of $f$?

To answer this question in full generality seems difficult, as experience has shown that in proving oscillatory integral estimates, one uses particular geometric features of the phase function, whereas the resulting estimate does not. We elaborate on this point with an example that will motivate some later arguments.

In seeking an oscillatory integral estimate, one expects that as $\lambda$ increases, the differences in $f(x)$ from moving in $x$ produce much larger differences in $\lambda f(x)$, which leads to greater cancellation in the integral of $e^{i\lambda f(x)}$. A natural condition to quantify this change in $f$ is by asking that some derivative is bounded below. This provides sufficient information to derive oscillatory integral estimates, but it also provides information as to where the values of $f$ occur, something that the integral itself does not see. For instance, in one dimension, the condition $f'\geq 1$ and $f(x)=c_1$ tells us that whenever $f(y)\geq c_1+c_2$, we must have $y\geq x+c_2$.

Now, let us consider the related problem of sublevel set estimates: given a real-valued function $f$ and a constant $c$ when do we have estimates of the form $|\{x:|f(x)-c|\leq\varepsilon\}|\leq C\varepsilon^\delta$? That this should be related follows from the intuition that in order to obtain good decay in our oscillatory integral estimates, we expect that $f$ should move around a lot, that is, it should not spend too long near any given value. If we can take our sublevel set estimates with $C$ independent of $c$, these estimates are quantifying that intuition precisely.

In fact, it is well known that sublevel set estimates follow from oscillatory integral estimates - see, for instance, Carbery-Christ-Wright \cite{carbery1999multidimensional}. For completeness, we record the precise observation here.

\begin{prop}\label{oscimpsub}
Let $(X,\mu)$ be any finite measure space, $0<\delta<1$ and $f:X\rightarrow\mathbb{R}$ a measurable function such that for all real non-zero $\lambda$, we have
$$\left|\int_Xe^{i\lambda f(x)}\,d\mu(x)\right|\leq A|\lambda|^{-\delta}.$$
Then for each $c\in\mathbb{R}$, we have
$$\mu(\{x\in X:|f(x)-c|\leq\varepsilon\})\leq C_\delta A\varepsilon^\delta$$
where $C_\delta$ depends only on $\delta$.
\end{prop}
\begin{proof}
Denote $U_{c,\varepsilon}=\{x\in X:|f(x)-c|\leq\varepsilon\}$. Fix $\phi\in C_c^\infty(\mathbb{R})$ with $0\leq\phi\leq 1$, $\phi(x)=1$ for $x\in [-1,1]$ and $\phi(x)=0$ for $x\notin[-2,2]$. Clearly we have
$$\mu(U_{c,\varepsilon})\leq\int_X \phi((f(x)-c)/\varepsilon)\,d\mu.$$
By the Fourier inversion formula, we may write the right hand side as
$$\int_X \int_{\mathbb{R}}\hat{\phi}(s)e^{2\pi i\xi{((f(x)-c)/\varepsilon)}}\,d\xi\,d\mu=\int_{\mathbb{R}}\hat{\phi}(\xi)e^{-2\pi i\xi c/\varepsilon}\int_X e^{2\pi i (\xi/\varepsilon)f(x)}\,d\mu\,d\xi.$$
By assumption, this is bounded by
$$A\varepsilon^\delta\int_{\mathbb{R}}|\hat{\phi}(\xi)||\xi|^{-\delta}\,d\xi$$
and noting that $\hat{\phi}$ is Schwartz and $|\xi|^{-\delta}$ is integrable near $0$ for $\delta<1$ gives the result.
\end{proof}

We remark that this is not true for $\delta=1$ - for instance, one sees that on the unit cube, the oscillatory integral with phase $f(x,y)=xy$ is bounded by a multiple of $\lambda^{-1}$, but the corresponding sublevel set has $|U_{0,\varepsilon}|\sim\varepsilon\log(1/\varepsilon)$ for small $\varepsilon$.

Now, in one dimension the classical van der Corput lemma captures the behaviour we expect from our intuition. It states that for a real valued function satisfying $|f^{(N)}|\geq 1$ on an interval $I=[a,b]$, with the additional assumption that $f'$ is monotone in the case $N=1$, the oscillatory integral $\int_I e^{i\lambda f(x)}\,dx$ is bounded by $C_N|\lambda|^{-1/N}$, where $C_N$ depends only on $N$ and not on $f$ or $I$. Uniformity of the constant means that we can scale to derive bounds for functions with $|f^{(N)}|\geq\alpha>0$.

A discussion of this result can be found in Carbery-Christ-Wright \cite{carbery1999multidimensional}. However, let us comment on the proof strategy used there. The strategy taken is to first prove the corresponding sublevel set estimate, and then use that to prove the oscillatory integral estimate. One sees that most of the work is contained in the sublevel set estimate, but we require a basic fact regarding the structure of sublevel set to complete the proof, namely that it is a union of finitely many intervals, the number of which is bounded by a number depending only on $N$ - this is obvious from our hypothesis.

The essential role of this structure appears in other proofs of van der Corput's lemma, for instance the standard proof given in Stein \cite{stein1993harmonic}. Thus in general, we anticipate difficulty in finding arguments that do not make use of any structural information of the phase.

The core of our main argument comes from the ideas described above. We derive good inclusion estimates for the sublevel sets of polynomials in terms of proximity to the roots, and we get a good inclusion estimate for the sublevel sets of $f$ composed with a polynomial in terms of the sublevel sets of $f$. It follows also that if the sublevel sets of $f$ have at most $N$ components, then the sublevel sets of its composition with a polynomial of degree $d$ have at most $C=C(N,d)$ components.

\section{Main Results}
We first consider the one dimensional case. By considering phases $f(x)=x^N$ and the monomials $P(x)=x^d$, we see that the best possible statement for any classes of phases and polynomials containing multiples of these is ``if the oscillatory integral with phase $f$ decays like $|\lambda|^{-\varepsilon}$, and $P$ is of degree $d$, then the oscillatory integral with phase $P(f)$ decays like $|\lambda|^{-\varepsilon/d}$". We derive such estimates with uniformity in parameters in many situations.

In the first instance, we can take $P$ to be a monic polynomial of degree $d$ (for clarity in our statements we will in fact assume $P'$ is monic). This situation covers all polynomials by scaling; however, if we are interested in the situation where the leading coefficient degenerates to $0$, the constant in such estimates blows up. Thus we shall also introduce a suitable class of normalised polynomials, and establish a result in this complementary situation.

\begin{defi}
Let $P\in\mathbb{R}[x]$ be given as $P(x)=a_dx^d+a_{d-1}x^{d-1}+\ldots+a_1x+a_0$. Suppose that $\max_j |a_{d-j}|=1$ and that this maximum is attained by some $j$ with $j\leq d/2$. For convenience we shall call such a polynomial a semi-non-degenerating (SND) polynomial.
\end{defi}
Such polynomials were studied by Kowalski \& Wright \cite{kowalski2012elementary}. We have chosen this name for them to reflect the fact that, as the higher order coefficients degenerate to $0$, at most half of the roots go off to infinity, whereas the remaining roots stay within a bounded set.

As alluded to in the introduction, it is important that we have a stable sublevel set inclusion for polynomials in terms of proximity to the roots; for these polynomials we have such an inclusion, this is Lemma \ref{inclu}, which is due to Kowalski \& Wright. They have noted that this inclusion can fail for non-SND polynomials.

\begin{thm}\label{calcdim1M}
Let $f$ be a smooth real-valued function on an interval $I=[a,b]$ and for some $N\geq 2$, suppose that $f^{(N)}$ is single-signed. Suppose that for some $\delta<1$ and all non-zero real $\lambda$, we have
$$\left|\int_I e^{i\lambda f(x)}\,dx\right|\leq A|\lambda|^{-\delta}$$
for a fixed $A\geq 1$. Then for any polynomial $P$ of degree $d\geq 2$ such that $P'$ is monic, and all non-zero $\lambda$, we have
$$\left|\int_I e^{i\lambda P(f(x))}\,dx\right|\leq C_{d,N,\delta}A^{1/(1-\delta)}|\lambda|^{-\delta/d}.$$
Here $C_{d,N,\delta}$ depends only on $d$, $N$ and $\delta$, and not on $f$ or $I$. In addition, the same estimate holds if $P'$ is SND and $|\lambda|\geq 1$.
\end{thm}

By single-signed, we mean either $f^{(N)}\geq 0$ or $f^{(N)}\leq 0$. Note that this assumption on $f^{(N)}$ is not quantitative, and in the proof it is only used to obtain structural statements about sublevel sets in the way outlined in the previous section.

We remark that we do not in fact directly use the assumption on the decay of the oscillatory integral, rather, we use its consequence given by Proposition \ref{oscimpsub}. However, we will see that such sublevel sets estimates imply oscillatory integral estimates so long as the above single-signed assumption is satisfied and N remains bounded, giving an approximate converse to Proposition 1.1
(see the discussion after Proposition \ref{derivpush}).

Observe that in the case where the number of sign changes of some $N^{\text{th}}$ derivative is bounded by $M$, we can simply subdivide our interval into parts where $f^{(N)}$ is single-signed, and get bounds of the form $C_{d,N}(M+1)A^{1/(1-\delta)}|\lambda|^{-\delta/d}$. 

This theorem applies in particular to the setting of van der Corput's lemma, except in the case where $|f'|\geq 1$ and $f'$ is monotone. Thus, the following result, which we shall prove simultaneously with the above theorem, is almost a corollary.

\begin{prop}\label{vdCgenM}
Let $f$ be a smooth real-valued function satisfying $|f^{(N)}|\geq 1$ on an interval $I=[a,b]$ for some $N\geq 1$, where if $N=1$ we assume in addition that $f'$ is monotone. Then for any polynomial $P$ of degree $d\geq 2$ such that $P'$ is monic or an SND polynomial, we have
$$\left|\int_I e^{i\lambda P(f(x))}\,dx\right|\leq C_{d,N}|\lambda|^{-1/Nd}.$$
for $|\lambda|\geq 1$. Here $C_{d,N}$ depends only on $d$ and $N$, and not on $f$ or $I$.
\end{prop}

Of course, where the condition that $|\lambda|\geq 1$ appears, it can be relaxed using trivial estimates if we allow $C_{d,N}$ to depend on $|I|$. It is then clear how this result scales when we replace $f$ or $P$ by a multiple.

Once these one-dimensional results are established, it is not too difficult to establish a higher-dimensional result in the situation where some mixed derivative is bounded below and we have additional higher order convexity assumptions.

First we give the natural generality of domains we shall admit for our functions $f$. Let $n\geq 2$ and $X\subseteq\mathbb{R}^n$ be a bounded, measurable set with the property that for $j=2,3\ldots,n$ and for each $(x_1,\ldots,x_{j-1},x_{j+1},\ldots,x_n)$ the set
$$\{y\in\mathbb{R}:(x_1,\ldots,x_{j-1},y,x_{j+1},\ldots,x_n)\in X\}$$
is a union of at most $A$ intervals (in particular this is true if $X$ is a union of at most $A$ convex sets).

For a multiindex $\beta\in\mathbb{N}_0^n$ we write $\partial^\beta=\partial_{x_1}^{\beta_1}\ldots\partial_{x_n}^{\beta_n}$ and $|\beta|=\beta_1+\ldots+\beta_n$.

\begin{thm}\label{highdim}
Let $X$, $A$ be as above. Let $f$ be a smooth real-valued function in a neighbourhood of $X$ satisfying $|\partial^{\beta}f|\geq 1$ on $X$. Let $N_2>\beta_2, N_3>\beta_3,\ldots,N_n>\beta_n$ such that $\partial^{(0,0,\ldots,0,N_n)}f,\partial^{(0,0,\ldots,N_{n-1},\beta_n)}f,\ldots,\partial^{(0,N_2,\beta_3,\ldots,\beta_n)}f$ are all single-signed.

Then for any polynomial $P$ of degree $d$ such that $P'$ is either monic or an SND polynomial, we have
$$\left|\int_X e^{i\lambda P(f(x))}\,dx\right|\leq C_{D,d,n,\beta,N_2,\ldots,N_n, A}|\lambda|^{-1/|\beta|d}.$$
Here $C_{D,d,\beta,N_2,\ldots,N_n}$ depends only on $D=\text{diam}(X)$, $d$, $\beta$, the $N_j$, $A$ and $n$, and not on $f$. The dependence on $D$ is bounded as long as the diameter remains bounded.
\end{thm}

The proof of this result essentially follows the analogous result of Carbery-Christ-Wright \cite{carbery1999multidimensional}. In that paper, they extend van der Corput's lemma to the higher dimensional setting where the natural assumption is that some mixed derivative is bounded below. They have an argument where they do not impose additional assumptions on the number of components in the slices, but do not obtain the sharp exponent, and one that does impose these assumptions but obtains the sharp exponent. We were unable to generalise the former to our setting, but we do still obtain the sharp exponent in this setting (which for a multiindex $\beta$ and functions $f$ satisfying $|\partial^\beta f|\geq 1$ is known to be a decay of $\lambda^{-1/|\beta|}$).

We were unable to obtain a complete analogue of Theorems \ref{calcdim1M}, that is to say, we are unable to replace the lower bound on the derivative in Theorem \ref{highdim} with the weaker assumption that some oscillatory integral estimate holds, but it would be interesting to find an alternative argument which handles this case.

Note that these results can be extended by limiting arguments, for instance, to distributions satisfying $|\partial^\beta f|\geq 1$ as distributions, as a consequence of the uniformity.

Let us make two more comments regarding the general themes of this paper. It is worth drawing the reader's attention to some related results by Kowalski \cite{kowalski2010comparative} concerning the asymptotics of one-dimensional oscillatory integral estimates after polynomial transformations of the phase.

An additional related question, motivated by the ease of deriving sublevel set estimates for products of functions satisfying sublevel set estimates (we discuss this in the next section), is ``Given two phase functions $f(x)$ and $g(y)$ satisfying oscillatory integral estimates, what can we say about the phase $f(x)g(y)$?"

This next result and the following discussion give an essentially complete answer to this question.

\begin{lem}\label{prodthm}
Let $f:X\rightarrow\mathbb{R}$ and $g:Y\rightarrow\mathbb{R}$ be measurable functions on finite measure spaces $(X,\mu)$ and $(Y,\nu)$. Suppose that for all non-zero real $\lambda$ we have
$$\left|\int_X e^{i\lambda f(x)}\,d\mu(x)\right|\leq A|\lambda|^{-\delta},\quad \left|\int_Y e^{i\lambda g(y)}\,d\nu(y)\right|\leq B|\lambda|^{-\delta'}$$
where $A,B,\delta,\delta'$ are positive constants with $\delta<\delta'<1$. Then for all non-zero real $\lambda$ we have
$$\left|\int_{X\times Y} e^{i\lambda f(x)g(y)}\,d\mu(x)\,d\nu(y)\right|\leq C_{\delta,\delta'}A(\nu(Y)+B)|\lambda|^{-\delta}$$
where $C_{\delta,\delta'}$ depends only on $\delta$ and $\delta'$.
\end{lem}

\begin{proof}
We can bound the inner integral by $A|\lambda g(y)|^{-\delta}$. Also, by Proposition \ref{oscimpsub} we have for each $\varepsilon>0$, $\nu(\{y\in Y:|g(y)|\leq\varepsilon\})\leq C_{\delta'}B\varepsilon^{\delta'}$. Thus $\nu(\{y\in Y:\varepsilon^{-\delta}\leq|g(y)|^{-\delta}\})\leq C_{\delta'}B\varepsilon^{\delta'}$. We obtain
\begin{align*}
\left|\int_{X\times Y} e^{i\lambda f(x)g(y)}\,d\mu(x)\,d\nu(y)\right|&\leq \int_Y A|\lambda g(y)|^{-\delta}\,d\nu(y)\\
&=A|\lambda|^{-\delta}\int_Y |g(y)|^{-\delta}\,d\nu(y)\\
&\leq A|\lambda|^{-\delta}\left(\nu(Y)+\sum\limits_{n=0}^\infty 2^{n+1}\nu_n\right)
\end{align*}
where $\nu_n:=\nu(\{y\in Y:2^n\leq|g(y)|^{-\delta}\leq 2^{n+1}\})\leq \nu(\{y\in Y:2^n\leq|g(y)|^{-\delta}\})\leq C_{\delta'}B(2^{-n/\delta})^{\delta'}$. Hence
\begin{align*}
\left|\int_{X\times Y} e^{i\lambda f(x)g(y)}\,d\mu(x)\,d\nu(y)\right|&\leq A|\lambda|^{-\delta}\left(\nu(Y)+C_{\delta'}B\sum\limits_{n=0}^\infty 2^{n+1}2^{-n\delta'/\delta}\right)\\
&\leq C_{\delta,\delta'}A(\nu(Y)+B)|\lambda|^{-\delta}
\end{align*}
as required.
\end{proof}

That this result is essentially the best possible can be seen by simple examples, paired with Proposition \ref{oscimpsub} so that it is sufficient to consider the sublevel set at height $0$. Consider $f(x)=x^k$ and $g(y)=y^j$ defined on $[0,1]$, for $k,j\geq 2$. For $k>j$, one can see by calculations of the measure of the corresponding sublevel set that the decay exponent of $\min(\delta,\delta')$ is optimal. When $j=k$, for which $\delta=\delta'$, we see that the exponent $\min(\delta,\delta')$ is not possible, for if the oscillatory integrals with phases $(xy)^j$ had decay $|\lambda|^{-1/j}$, we would obtain the estimates $|\{(x,y)\in[0,1]^2:(xy)^j\leq\varepsilon\}|\leq\varepsilon^{1/j}$, which are false as noted in the introduction.

\textbf{Possible extensions. }We remark that the robustness of our arguments allows us to replace $P$ with many things besides polynomials, though we shall not formulate any precise results. Our arguments will make sense in the setting where the phase $f$ is composed with any function $P$ satisfying a few basic principles. First, that $P'$ is defined except at boundedly many points (for we can delete sublevel sets of $f$ from the domain around the values for which it is not defined), and that we have a ``good" inclusion for the sublevel sets of $P'$ in terms of structurally well understood sets, such as intervals. We also ask that $P'$ changes monotonicity boundedly many times.

A concrete example is given by $P(x)=|x|^s$, $s>1$. Running the main arguments with only trivial modifications yields statements analogous to the main theorems, with uniform constants and the analogous decay rates where $d$ is replaced by $s$.

\textbf{Notation. }We will often introduce constants $C$, $C'$, $C''$, etc in our proofs without explicit reference and as necessary when we wish to distinguish their values. Their dependencies will be denoted with a subscript.

\section{Proofs of the main theorems}
\subsection{Structural statements for sublevel sets}
In a general measure space $(X,\mu)$, we consider the sublevel set estimates $\mu(\{x\in X:|f(x)|\leq\varepsilon\})\leq C\varepsilon^p$. It is clear from rewriting this expression that the best possible constant $C$ for which this holds is the standard weak $L^p$ norm of $1/f$. This perspective naturally leads one to consider which inequalities relating various weak $L^p$ norms of functions have meaningful statements in terms of sublevel sets. As we are considering the reciprocal of a function, it seems best to concern ourselves with operations that respect reciprocals, such as products.

This leads us to consider the well-known analogue of H\"older's inequality for weak $L^p$ spaces (see, for instance, Grafakos \cite{grafakos2008classical}). This says that when $1/p=\sum_{j=1}^k1/p_i$, we have $\|f_1\ldots f_k\|_{L^{p,\infty}}\leq C_{p_1,\ldots,p_k}\|f_1\|_{L^{p_1,\infty}}\ldots\|f_k\|_{L^{p_k,\infty}}$. In terms of sublevel sets, this says that when we know the $f_j$ satisfy sublevel set estimates with constant $C_i$ and exponents $p_i$, then the product $f_1\ldots f_k$ satisfies sublevel set estimates with exponent $p$ given by $1/p=\sum_{j=1}^k1/p_i$, and constant controlled by $C_1\ldots C_k$.

This tells us something about the size of sublevel sets in the product, but does not tell us anything about their structure. However, following the ideas in standard proofs of H\"older's inequality and its weak $L^p$ analogue easily yield structural statements; for instance, we have:

\begin{prop}
Let $f_i:X\rightarrow\mathbb{C}$, then we have the inclusion
$$\{x\in X:|f_1\ldots f_k(x)|\leq\varepsilon\}\subseteq\cup_{i=1}^k\{x\in X:|f_i(x)|\leq(k\varepsilon/p_i)^{1/p_i}\}$$
for any $p_i$ satisfying $\sum_{i=1}^k1/p_i=1$. Furthermore, if the $f_i$ satisfy the estimates
$$\mu(\{x\in X:|f_i(x)|\leq\varepsilon\})\leq C_i\varepsilon^{\delta_i}$$
for some positive numbers $C_i$ and $\delta_i$, then we have
$$\mu(\{x\in X:|f_1\ldots f_k(x)|\leq\varepsilon\}|\leq C\varepsilon^{\delta}$$
where
$$\delta:=\left(\sum_{i=1}^k\delta_i^{-1}\right)^{-1},\quad C=\sum\limits_{i=1}^kC_i\left(\frac{k\delta}{\delta_i}\right)^\delta.$$
\end{prop}

We have also included the corresponding estimate, which is a little weaker than the result from the weak $L^p$ H\"older inequality (by way of the inequality of arithmetic and geometric mean), simply to indicate some natural choices of $p_i$ when sublevel set estimates are known for the $f_i$.

\begin{proof}
Recall Young's inequality: For $p_i\geq 1$ satisfying $\sum_{i=1}^k1/p_i=1$ and any numbers $a_i\geq 0$, we have
$$a_1\dots a_k\leq \sum\limits_{i=1}^k\frac{a_i^{p_i}}{p_i}.$$
We use this as follows. Observe that $|f_1\ldots f_k(x)|\leq\varepsilon$ if and only if $\varepsilon^{-1}\leq\prod_{i=1}^k|f_i(x)|^{-1}$. Applying Young's inequality to the right hand side gives
$$\varepsilon^{-1}\leq \sum\limits_{i=1}^k\frac{|f_i(x)|^{-p_i}}{p_i}.$$
Then at least one of the $|f_i(x)|^{-p_i}/p_i$ is greater than or equal to $(k\varepsilon)^{-1}$, so we deduce the inclusion
$$\{x\in X:|f_1\ldots f_k(x)|\leq\varepsilon\}\subseteq\cup_{i=1}^k\{x\in\Omega:|f_i(x)|\leq(k\varepsilon/p_i)^{1/p_i}\}$$
as required. We can also bound the measure of the right hand side by the sum of the measure of each set in the union and apply the known sublevel set estimates, giving
$$\mu(\{x\in\Omega:|f(x)|\leq\varepsilon\})\leq\sum\limits_{i=1}^kC_i\left(\frac{k\varepsilon}{p_i}\right)^{\delta_i/p_i}.$$
We optimise the exponent for $\varepsilon$ by choosing the $p_i$ to be such that all the $\delta_i/p_i$ are equal to some $\delta$, subject to the constraint $\sum_{i=1}^kp_i^{-1}=1$. Substituting $1/p_i=\delta/\delta_i$, we see that we should take $\delta:=\left(\sum_{i=1}^k\delta_i^{-1}\right)^{-1}$, and accordingly set $p_i=\delta_i/\delta$.
\end{proof}

Observe that in the case where all $\delta_i$ are equal, $\delta/\delta_i=1/p_i=1/k$, so $C$ just becomes the sum of the $C_i$. The corresponding sublevel set inclusion becomes
$$\{x\in X:|f_1\ldots f_k(x)|\leq\varepsilon\}\subseteq\cup_{i=1}^k\{x\in X:|f_i(x)|\leq\varepsilon^{1/k}\}.$$

This has a useful consequence for polynomials of one variable in $\mathbb{C}$ (and by intersecting with the real line, also for $\mathbb{R}$). Suppose $P(z)=(z-z_1)\dots(z-z_d)$ is a monic polynomial. We obtain:

\begin{prop}\label{monicinc}
When $P$ is a monic polynomial of degree $d$ with roots $z_i$, we have
$$\{x:|P(x)|\leq\varepsilon^d\}\subseteq\cup_{i=1}^d\{x:|x-z_i|\leq\varepsilon\}.$$
\end{prop}

Let us make some remarks on this product sublevel set inclusion. Note that it is not translation invariant, that is, it does not work for sublevel sets not at height $0$. For instance, we cannot obtain estimates $\{x:|f_1\ldots f_k(x)-c_0|\leq\varepsilon\}$ given estimates for $\{x:|f_i(x)-c|\leq\varepsilon\}$ for each $c$. To see this, consider the example $f_1(x)=x$ and $f_2(x)=1/x$ on the interval $[1,2]$. Since both have bounded below derivative, $|\{x:|f_i(x)-c|\leq\varepsilon\}|\leq C\varepsilon$ for some $C$ independent of $i$ and $c$, however,  $|\{x:|f_1f_2(x)-1|\leq\varepsilon\}|=1$.

Noting Proposition \ref{oscimpsub}, we thus cannot expect oscillatory integral estimates to hold for $f_1(x)f_2(x)$ when they hold for $f_1$ and $f_2$, so in general there is no calculus giving oscillatory integral estimates for $P(f_1(x),f_2(x))$ where $P$ is a polynomial of two variables. This is in contrast with the more optimistic conclusions of Theorem \ref{prodthm} for when $f_1$ and $f_2$ are functions of different variables.

Nevertheless, Proposition \ref{monicinc} is sufficient for the proofs of the monic cases of Theorems \ref{calcdim1M} and \ref{highdim}. For the other results, we need an analogous structural statement for SND polynomials, which is the following result of Kowalski \& Wright \cite{kowalski2012elementary}:

\begin{lem}\label{inclu}
For an SND polynomial and $\varepsilon<1$, we have the inclusion
$$\{x:|P(x)|\leq\varepsilon^d\}\subseteq\cup_{i=1}^d\{x:|x-z_j|\leq B_d\varepsilon\}$$
where $z_j$ are the (possibly complex) roots of $P$, and $B_d$ is a constant depending only on $d$.
\end{lem}
This result is $(25)$ in that paper. They also observe that this inclusion can fail when $P$ is not an SND polynomial by considering the family of polynomials $P(x)=\eta x^{k-1}(s-\eta^{-1/k})^k$, where $0<\eta<1$.

\subsection{The one-dimensional case}
Both of the situations described in Theorem \ref{calcdim1M} will be addressed simultaneously, and we will highlight the differences as they occur. We will also note the differences in adapting to the setting of van der Corput's lemma as in Proposition \ref{vdCgenM}; these slight modifications will be used in the later proof of Theorem \ref{highdim}.

In all cases, the assumption that some higher order derivative is single-signed means that we can divide $I$ into boundedly many intervals, depending on $N$, on each of which $f$ and all its derivatives up to order $N-1$ are monotone. Hence without loss of generality, we may assume that $f$ and $f'$ are monotone on $I$.

Furthermore, we may divide $\mathbb{R}$ into boundedly many intervals $J$, depending on $d$, on each of which $P'$ is monotone. Further dividing $I$ by intersecting with each $f^{-1}(J)$, we may assume without loss of generality that $P'(f)$ is monotone on $I$.

Our assumptions also allow us to assume that the sublevel sets $\{x\in I:|f(x)-c|\leq\varepsilon\}$ are unions of boundedly many intervals, depending on $N$. Now, using the structural results for the sublevel sets of monic polynomials from the preliminaries, we have
$$\{x\in I:|P'(f(x))|\leq\varepsilon^{d-1}\}\subseteq\cup_{j=1}^{d-1}\{x\in I:|f(x)-z_j|\leq\varepsilon\}$$
for all $\varepsilon>0$ where $z_j$ are the roots of $P'$ (if any $z_j$ is non-real we may clearly replace is with its real part). Similarly, using the result for SND polynomials we have
$$\{x\in I:|P'(f(x))|\leq\varepsilon^{d-1}\}\subseteq\cup_{j=1}^{d-1}\{x\in I:|f(x)-z_j|\leq B_{d-1}\varepsilon\}$$
provided $\varepsilon\leq 1$. We will estimate the oscillatory integral over the set on the right by its measure. For convenience we shall write $B_{d-1}$ when working in the monic case also.

We have by Proposition \ref{oscimpsub} that each of the sets on the right hand side is bounded by $C_\delta A(B_{d-1}\varepsilon)^\delta$, hence
$$|\cup_{j=1}^{d-1}\{x\in I:|f(x)-z_j|\leq B_{d-1}\varepsilon\}|\leq C'_{d,\delta}A\varepsilon^\delta.$$
For reference we note that in the $|f^{(N)}|\geq 1$ case, we directly have the estimate
$$|\cup_{j=1}^{d-1}\{x\in I:|f(x)-z_j|\leq B_{d-1}\varepsilon\}|\leq C'_{d,N}\varepsilon^{1/N}.$$

Since the right hand side of the above set inclusion is a union of $d-1$ sets which are unions of boundedly many intervals (depending on $N$), it is a union of boundedly many intervals (depending on $d$ and $N$), and therefore so is its complement in $I$. We now consider the contribution to the oscillatory integral over the complement of this union.

Because of the set inclusion above, on each of these intervals in the complement, we have $|P'(f(x))|\geq\varepsilon^{d-1}$ and $P'(f(x))$ is monotone, and $f'$ is monotone. Each such interval splits into at most $3$ intervals on each of which either $|f'|<r$ or $|f'|\geq r$, where $r$ is to be chosen later. Note that in the $|f'|\geq 1$ case, we will take $r\leq 1$ so that the former is empty and we ignore its contribution.

Let us summarise where we are at. Up to a constant depending only on parameters $\delta$, $d$, and $N$, the oscillatory integral is bounded by the sum of three terms. The first term is $A\varepsilon^\delta$ (or $\varepsilon^{1/N}$ in the $|f^{(N)}|\geq 1$ case). The remaining two terms are integrals over an interval - which we shall relabel as $I$ in each case - on which $|P'(f(x))|\geq\varepsilon^{d-1}$ and $P'(f(x))$ is monotone, $f'$ is monotone and either $|f'|<r$ or $|f'|\geq r$.

In the first case, we will use the following simple bound for the measure of an interval on which $|f'|\leq r$:

\begin{prop}\label{derivpush}
Let $I=[a,b]$ be an interval on which $|f'|\leq r$, and suppose that $|\{x\in I:|f(x)-c|\leq\alpha\}|\leq B\alpha^\delta$ for some $\delta<1$ and each positive $\alpha$ and real $c$. Then $|I|\leq(B/2^\delta)^{1/(1-\delta)}r^{\delta/(1-\delta)}$.
\end{prop}

We remark therefore that the estimates $|\{x\in I:|f(x)-c|\leq\alpha\}|\leq B\alpha^\delta$ imply the corresponding oscillatory integral estimates (with uniformity in $f$ so long as $B$ remains bounded) so long as the single-signed assumption is satisfied and $N$ remains bounded. This follows from the proposition and a standard integration by parts argument as in the proof of van der Corput's lemma (see, for instance, Stein \cite{stein1993harmonic}, we shall also use a modification of this below). Hence we obtain an approximate converse to Proposition \ref{oscimpsub}.

\begin{proof}
Let $v=(a+b)/2$. If there were $x\in I$ such that $|f(x)-f(v)|>|I|r/2$, then $|f(x)-f(v)|>|x-v|r$, hence by the mean value theorem there is a point in $I$ with $|f'|>r$, a contradiction. Thus $I\subseteq \{x\in I:|f(x)-f(v)|\leq |I|r/2\}$. Applying the sublevel set estimate yields $|I|\leq B(|I|r/2)^\delta$, hence
$|I|\leq (B/2^\delta)^{1/(1-\delta)}r^{\delta/(1-\delta)}$.
\end{proof}
Using the result along with Proposition \ref{oscimpsub} allows us to estimate in the single-signed $N^{\text{th}}$ derivative case by
$$(C_\delta A/2^\delta)^{1/(1-\delta)}r^{\delta/(1-\delta)}=C''_{\delta}A^{1/(1-\delta)}r^{\delta/(1-\delta)}.$$
Note that when $\delta=1/N$, this is $C''_NA^{N/(N-1)}r^{1/(N-1)}$, so in the $|f^{(N)}|\geq 1$ case this gives the expected exponent - although in the situation of Proposition \ref{vdCgenM} we have the bound $C''_Nr^{1/(N-1)}$ directly through standard sublevel set estimates.

To deal with $|f'|\geq r$, we will use a slight modification of the standard integration by parts argument from the proof of van der Corput's lemma. Write $I=[a,b]$, we have
\begin{align*}
\int_a^b e^{i\lambda P(f(x))}\,dx&=\int_a^b \frac{1}{\lambda P'(f)f'}\frac{d}{dx}e^{i\lambda P(f(x))}\,dx\\
&=\frac{1}{\lambda}\left(\left.\frac{e^{i\lambda P(f(x))}}{P'(f)f'}\right|_a^b-\int_a^b \frac{d}{dx}\left(\frac{1}{P'(f)f'}\right)e^{i\lambda P(f(x))}\,dx\right)
\end{align*}
The boundary terms are bounded in absolute value by $2(r\varepsilon^{d-1})^{-1}$. Also, we have
\begin{align*}
\left|\int_a^b \frac{d}{dx}\left(\frac{1}{P'(f)f'}\right)e^{i\lambda P(f(x))}\,dx\right|&\leq\int_a^b \left|\frac{d}{dx}\left(\frac{1}{P'(f)f'}\right)\right|\,dx\\
&\leq \frac{1}{r}\int_a^b \left|\left(\frac{1}{P'(f)}\right)'\right|\,dx\\
&\quad+\frac{1}{\varepsilon^{d-1}}\int_a^b \left|\left(\frac{1}{f'}\right)'\right|\,dx
\end{align*}
Where the second inequality follows from the product rule and our bounds on $f'$ and $P'(f)$. Since each of the derivatives appearing in the integrals is of a monotone function (since $f'$ and $P'(f)$ are single-signed and monotone, so is their reciprocal), they are single-signed, and so we can pull the absolute values out of the integral and calculate. Using the bounds on $f'$ and $P'(f)$, we get that this integral is bounded by $4(r\varepsilon^{d-1})^{-1}$, so in total we have a bound of $6(r|\lambda|\varepsilon^{d-1})^{-1}$.

Altogether then, the oscillatory integral is bounded (up to a uniform constant) by
$$A\varepsilon^\delta+A^{1/(1-\delta)}r^{\delta/(1-\delta)}+(r|\lambda|\varepsilon^{d-1})^{-1}.$$
In the setting of van der Corput's lemma, we have
$$\varepsilon^{1/N}+r^{1/(N-1)}+(r|\lambda|\varepsilon^{d-1})^{-1}$$
where the middle term is omitted if $N=1$.

To conclude, set $\varepsilon=|\lambda|^{-1/d}$ and $r=|\lambda|^{-(1-\delta)/d}$ to obtain the desired result, where we impose $|\lambda|\geq 1$ in the SND case. Note $A\geq 1$ implies $1+A+A^{1/(1-\delta)}\leq 3A^{1/(1-\delta)}$.

In the setting of van der Corput's lemma, we set $\varepsilon=|\lambda|^{-1/d}$ and $r=|\lambda|^{-(N-1)/Nd}$ to obtain the desired result, where we impose $|\lambda|\geq 1$ in the SND case (so that $\varepsilon\leq 1$ in order to use Lemma \ref{inclu}) and in the $N=1$ case (so that $r=1$ and we ignore the $r^{1/(N-1)})$ term).

\subsection{The higher-dimensional case}
Here we give the proof of Theorem \ref{highdim}. We will run essentially the same argument multiple times, so we begin by outlining the order of the logical steps, then present the argument, with the minor differences indicated.

We first prove the result for $P(x)=x$ as in Carbery-Christ-Wright \cite{carbery1999multidimensional} by induction on the dimension. The proof for the base $n=2$ case proceeds exactly as the inductive step, so we shall make no distinction. The same proof used for the inductive step is then also used with the one-dimensional analysis of the previous section to prove the full result.

We may assume $|\lambda|\geq 1$, since trivial estimates yield the desired bound for $|\lambda|<1$. Write $x=(x',x_n)$, $\beta=(\beta',\beta_n)$. We will use $X'=X'(x_n)$ and $X_n=X_n(x')$ to denote the slices $\{x':x\in X\}$ for fixed $x_n$ and $\{x_n:x\in X\}$ for fixed $x'$ respectively.

For a parameter $\gamma>0$, split the integral into two parts, one over the set where $|\partial_{x_n}^{\beta_n}f|\geq\gamma$ and one over its complement. For the former, for fixed $x'$ consider the integral in $x_n$. There are at most $A$ intervals in this slice of $X$, in each of these intervals the set $\{x_n:|\partial_{x_n}^{\beta_n}f|\geq\gamma\}$ is a union of a number of intervals bounded by a number depending only on $N_n$ and $\beta_n$. In the case $\beta_n=1$, we may subdivide into a number of intervals depending on $N_n$ to ensure that on each, $\partial_{x_n}^{\beta_n}f$ is monotone.

In the $P(x)=x$ case we obtain
\begin{align*}
\left|\int_{X\cap\{x:|\partial_{x_n}^{\beta_n}f|\geq\gamma\}}e^{i\lambda P(f(x))}\,dx\right|&\leq\int\left|\int_{X_n\cap\{x_n:|\partial_{x_n}^{\beta_n}f|\geq\gamma\}}e^{i\lambda P(f(x',x_n))}\,dx_n\right|\,dx'\\
&\leq \int C_{D,d,\beta_n,N_n,A}K(\lambda,\gamma)\,dx'\\
&\leq C_nD^{n-1}C_{D,d,\beta_n,N_n,A}K(\lambda,\gamma)\\
&=C'_{D,d,n,\beta_n,N_n,A}K(\lambda,\gamma)
\end{align*}
where the final inequality comes from bounding the measure of the projection of $X$ onto the first $n-1$ coordinates and $K(\lambda,\gamma)$ is $|\lambda\gamma|^{-1/\beta_n}$ (by the usual van der Corput estimate).

In the general case, the same argument holds with a different value of $K(\lambda,\gamma)$ via an identical argument as in the proofs of the one-dimensional results. The value is as follows. 

With parameters $\varepsilon$ and $r$ to be chosen later, when $\beta_n>1$ we may take $K(\lambda,\gamma)=\gamma^{-1/\beta_n}\varepsilon^{1/\beta_n}+\gamma^{-1/(\beta_n-1)}r^{1/(\beta_n-1)}+(r|\lambda|\varepsilon^{d-1})^{-1}$, with $0<\varepsilon\leq 1$ in the SND case (in order to apply Lemma \ref{inclu}). If $\beta_n=1$, we take $K(\lambda,\gamma)=\gamma^{-1}\varepsilon+(\gamma|\lambda|\varepsilon^{d-1})^{-1}$, where $0<\varepsilon\leq 1$ in the SND case and we have taken $r=\gamma$ in the argument of the one dimensional case so that the set where $|\partial_{x_n}^{\beta_n}f|<\gamma$ is empty.

To bound the integral over the set where $|\partial_{x_n}^{\beta_n}f|<\gamma$, we simply bound by the measure, that is, by
$$\int\int_{X'\cap\{x':|\partial_{x_n}^{\beta_n}f(x',x_n)|<\gamma\}}\,dx'\,dx_n.$$
Provided $|\beta'|>1$, the $n-1$ case of the theorem applied to $\partial_{x_n}^{\beta_n}f$ with derivative $\partial^{\beta'}$ and polynomial $P(x)=x$, noting the diameter of $X'$ is bounded by the diameter of $X$, paired with Proposition \ref{oscimpsub}, gives estimates of $C_{D,d,n,\beta',N_2,\ldots,N_{n-1}, A}\gamma^{1/|\beta'|}$.

If $|\beta'|=1$, the same remains true. Then for some $j$ we have $|\partial_{x_j}\partial_{x_n}^{\beta_n}f|\geq 1$. We split the $x'$ integral over $x''=(x_1,\ldots,x_{j-1},x_{j+1},\ldots,x_{n-1})$ and $x_j$, so by first performing the integral in the $x_j$, we are to bound the sublevel set of a function having first derivative bounded below by $1$ in a union of at most $A$ intervals, which is at most $2A\gamma$, and then integrating in the other $n-2$ variables results in multiplying by at most $C_nD^{n-2}$, hence again we obtain the desired result.

Now integrating in $x_n$ contributes a factor of at most $D$, so the combined contribution over both sets is
$$\left|\int_X e^{i\lambda P(f(x))}\,dx\right|\leq C_{D,d,n,\beta,N_2,\ldots,N_n, A}(K(\lambda,\gamma)+\gamma^{1/|\beta'|}).$$
In the $P(x)=x$ case, we have $K(\lambda,\gamma)+\gamma^{1/|\beta'|}=|\lambda\gamma|^{-1/\beta_n}+\gamma^{1/|\beta'|}$, and setting $\gamma=|\lambda|^{-|\beta'|/|\beta|}$ gives $|\lambda|^{-1/|\beta|}$ as required.

For the general polynomial case, first if $\beta_n>1$, set $\gamma=|\lambda|^{-|\beta'|/d|\beta|}$, $\varepsilon=|\lambda|^{-1/d}$ (which is at most $1$ for $|\lambda|\geq 1$) and $r=|\lambda|^{-1/d+1/d|\beta|}$. Then 
\begin{align*}
K(\lambda,\gamma)+\gamma^{1/|\beta'|}&=(\varepsilon/\gamma)^{1/\beta_n}+(r/\gamma)^{1/(\beta_n-1)}+(r|\lambda|\varepsilon^{d-1})^{-1}+|\lambda|^{-1/d|\beta|}\\
&=(|\lambda|^{-\beta_n/d|\beta|})^{1/\beta_n}+(|\lambda|^{-(\beta_n-1)/d|\beta|})^{1/(\beta_n-1)}\\
&\quad+(r|\lambda|\varepsilon^{d-1})^{-1}+|\lambda|^{-1/d|\beta|}\\
&=3|\lambda|^{-1/d|\beta|}+(r|\lambda|\varepsilon^{d-1})^{-1}=4|\lambda|^{-1/d|\beta|}.
\end{align*}

If $\beta_n=1$, so $|\beta'|=|\beta|-1$, we set $\gamma=|\lambda|^{-(|\beta|-1)/d|\beta|}$, $\varepsilon=|\lambda|^{-1/d}$ and obtain
\begin{align*}
K(\lambda,\gamma)+\gamma^{1/|\beta'|}&=(\varepsilon/\gamma)+(\gamma|\lambda|\varepsilon^{d-1})^{-1}+\gamma^{1/|\beta'|}\\
&=2|\lambda|^{-1/d|\beta|}+(\gamma|\lambda|\varepsilon^{d-1})^{-1}=3|\lambda|^{-1/d|\beta|}
\end{align*}
as desired. This completes the proof of Theorem \ref{highdim}.

\textit{Acknowledgements.} The author is supported by a UK EPSRC scholarship at the Maxwell Institute Graduate School. The author would like to thank Prof. James Wright for many helpful discussions and comments, in particular the technique for estimating oscillatory integrals when the derivative is a product of two bounded below monotone functions is due to him.

\bibliographystyle{abbrv}
\bibliography{Bibliography1}

John Green,\\ Maxwell Institute of Mathematical Sciences and the School of Mathematics,\\ University of Edinburgh,\\ JCMB, The King’s Buildings,\\ Peter Guthrie Tait Road,\\ Edinburgh, EH9 3FD,\\ Scotland\\ Email: \texttt{J.D.Green@sms.ed.ac.uk}
\end{document}